\pdfoutput=1
\RequirePackage{ifpdf}
\ifpdf 
\documentclass[pdftex]{sigma}
\else
\documentclass{sigma}
\fi

\usepackage[all]{xy}

\numberwithin{equation}{section}

\newtheorem{Theorem}{Theorem}[section]
\newtheorem{Lemma}[Theorem]{Lemma}
\newtheorem{Proposition}[Theorem]{Proposition}
 { \theoremstyle{definition}
\newtheorem{Definition}[Theorem]{Definition}
\newtheorem{Example}[Theorem]{Example}
\newtheorem{Remark}[Theorem]{Remark} }

\newcommand{\nmod}{\textsf{-}\mathsf{mod}}
\newcommand{\modn}{\mathsf{mod}\textsf{-}}
\newcommand{\Mull}{\mathsf{M}}
\newcommand{\HC}{\mathsf{R}}
\newcommand{\rHC}{{}^*\mathsf{R}}
\def\simto{\overset{\sim}\to}

\begin{document}

\allowdisplaybreaks

\newcommand{\arXivNumber}{1706.04743}

\renewcommand{\thefootnote}{}

\renewcommand{\PaperNumber}{007}

\FirstPageHeading

\ShortArticleName{Alvis--Curtis Duality and a Generalized Mullineux Involution}

\ArticleName{Alvis--Curtis Duality for Finite General Linear Groups and a Generalized Mullineux Involution\footnote{This paper is a~contribution to the Special Issue on the Representation Theory of the Symmetric Groups and Related Topics. The full collection is available at \href{https://www.emis.de/journals/SIGMA/symmetric-groups-2018.html}{https://www.emis.de/journals/SIGMA/symmetric-groups-2018.html}}}

\Author{Olivier DUDAS~$^\dag$ and Nicolas JACON~$^\ddag$}

\AuthorNameForHeading{O.~Dudas and N.~Jacon}

\Address{$^\dag$~Universit\'e Paris Diderot, UFR de Math\'ematiques, B\^atiment Sophie Germain,\\
\hphantom{$^\dag$}~5 rue Thomas Mann, 75205 Paris CEDEX 13, France}
\EmailD{\href{mailto:olivier.dudas@imj-prg.fr}{olivier.dudas@imj-prg.fr}}

\Address{$^\ddag$~Universit\'e de Reims Champagne-Ardenne, UFR Sciences exactes et naturelles,\\
\hphantom{$^\ddag$}~Laboratoire de Math\'ematiques EA 4535, Moulin de la Housse BP 1039, 51100 Reims, France}
\EmailD{\href{mailto:nicolas.jacon@univ-reims.fr}{nicolas.jacon@univ-reims.fr}}

\ArticleDates{Received June 17, 2017, in f\/inal form January 22, 2018; Published online January 30, 2018}

\Abstract{We study the ef\/fect of Alvis--Curtis duality on the unipotent representations of $\mathrm{GL}_n(q)$ in non-def\/ining characteristic $\ell$. We show that the permutation induced on the simple modules can be expressed in terms of a generalization of the Mullineux involution on the set of all partitions, which involves both $\ell$ and the order of $q$ modulo $\ell$. }

\Keywords{Mullineux involution; Alvis--Curtis duality; crystal graph; Harish-Chandra theo\-ry}

\Classification{20C20; 20C30; 05E10}

\renewcommand{\thefootnote}{\arabic{footnote}}
\setcounter{footnote}{0}

\section{Introduction}

Let $\mathfrak{S}_n$ be the symmetric group on $n$ letters. It is well known that the complex irreducible representations of $\mathfrak{S}_n$ are naturally labelled by the set of partitions of~$n$. Tensoring with the sign representation induces a permutation of the irreducible characters which corresponds to the conjugation of partitions. An analogous involution can be considered for representations in positive characteristic $p > 0$. In this case, the irreducible representations are parametrized by $p$-regular partitions and the permutation induced by tensoring with the sign representation has a more complicated combinatorial description. The explicit computation of this involution $\Mull_p$ was f\/irst conjectured by Mullineux in \cite{Mu} and proved by Ford--Kleshchev in \cite{ForKle}. Their result was later generalized to representations of Hecke algebras at a root of unity by Brundan \cite{Brun} with a view to extending the def\/inition of $\Mull_p$ to the case where $p$ is any positive integer.

For representations of a f\/inite group of Lie type $G$, the Alvis--Curtis duality $\mathsf{D}_G$ functor (see \cite{DL82}) provides an involution of the same nature. For example, the Alvis--Curtis dual of a~complex unipotent character of $\mathrm{GL}_n(q)$ parametrized by a partition is, up to a sign, the character parametrized by the conjugate partition. Unlike the case of symmetric groups, the duality $\mathsf{D}_G$ does not necessarily map irreducible representations to irreducible representations, but only to complexes of representations. Nevertheless, Chuang--Rouquier showed in~\cite{ChRou} how to single out a~specif\/ic composition factor in the cohomology of these complexes, yielding an involution $\mathsf{d}_G$ on the set of irreducible representations of~$G$ in non-def\/ining characteristic $\ell \geq 0$.

The purpose of this paper is to explain how to compute this involution $\mathsf{d}_G$ using the Harish-Chandra theory and the representation theory of Hecke algebras, see Theorem \ref{thm:main}. We illustrate our method on the irreducible unipotent representations of $\mathrm{GL}_n(q)$, which are parametrized by partitions of $n$. This yields an explicit involution $\Mull_{e,\ell}$ on the set of partitions of $n$ which depends both on $\ell$ and on the order $e$ of $q$ modulo $\ell$ (with the convention that $e=\ell$ if $\ell \,|\, q-1$), see Theorem~\ref{thm:main2}.

\begin{Theorem} Assume that $\ell \nmid q$. Let $S(\lambda)$ be the simple unipotent module of $\mathrm{GL}_n(q)$ over $\mathbb{F}_\ell$ parametrized by the partition $\lambda$. Then \begin{gather*}\mathsf{d}_G(S(\lambda)) = S(\Mull_{e,\ell}(\lambda)).\end{gather*}
\end{Theorem}

The involution $\Mull_{e,\ell}$ is a generalization of the original Mullineux involution since it is def\/ined on the set of all partitions, and coincides with $\Mull_e$ on the set of $e$-regular partitions. Such a~generalization already appeared in a work of Bezrukavnikov~\cite{Be} and Losev~\cite{Lo} on wall-crossing functors for representations of rational Cherednik algebras, but in the case where $\ell \gg 0$.

Kleshchev showed in \cite{Kle} that the Mullineux involution $\Mull_p$ can also be interpreted in the language of crystals for Fock spaces in af\/f\/ine type~$A$, which are certain colored oriented graphs whose vertices are labeled by partitions. More precisely, the image by $\Mull_p$ of a $p$-regular partition~$\lambda$ is obtained by changing the sign of each arrow in a path from the empty partition to $\lambda$ in the graph. We propose a def\/inition of several higher level crystal operators on the Fock space which give a similar description for our generalized involution $\Mull_{e,\ell}$, see Proposition \ref{prop:main3}.

The paper is organized as follows. In Section~\ref{section2} we introduce the Alvis--Curtis duality for f\/inite reductive groups and show how to compute it within a given Harish-Chandra series using a similar duality for the corresponding Hecke algebra. Section~\ref{section3} illustrates our method in the case of f\/inite general linear groups. We give in Section~\ref{sec:hcgln} the def\/inition of a generalized version of the Mullineux involution and show in Theorem~\ref{thm:main2} that it is the shadow of the Alvis--Curtis duality for $\mathrm{GL}_n(q)$. The f\/inal section is devoted to an interpretation of our result in the context of the theory of crystal graphs.

\section{Alvis--Curtis duality}\label{section2}

In this section we investigate the relation between the Alvis--Curtis duality for a f\/inite reductive group within a Harish-Chandra series and a similar duality in the Hecke algebra associated to the series.

\subsection{Notation}\label{sec:notation}
Let $\mathbf{G}$ be a connected reductive algebraic group def\/ined over an algebraic closure of a f\/inite f\/ield of characteristic $p$, together with an endomorphism $F$, a power of which is a Frobenius endomorphism. Given an $F$-stable closed subgroup $\mathbf{H}$ of $\mathbf{G}$, we will denote by $H$ the f\/inite group of f\/ixed points $\mathbf{H}^F$. The group $G$ is a \emph{finite reductive group}.

We will be interested in the modular representations of $G$ in non-def\/ining characteristic. We f\/ix a prime number $\ell$ dif\/ferent from $p$ and an $\ell$-modular system $(K,\mathcal{O},k)$ which is assumed to be large enough for~$G$, so that the algebras $KG$ and $kG$ split. Throughout this section $\Lambda$ will denote any ring among $K$, $\mathcal{O}$ and~$k$.

Given a f\/inite-dimensional $\Lambda$-algebra $A$, we denote by $A\nmod$ (resp. $\mathsf{mod}\text{-}A$) the category of f\/inite-dimensional left (resp.\ right) $A$-modules. The corresponding bounded derived category will be denoted by $D^b(A\nmod)$ (resp. $D^b(\mathsf{mod}\text{-}A)$) or simply $D^b(A)$ when there is no risk of confusion. We will identify the Grothendieck group of the abelian category $A\nmod$ with the Grothendieck group of the triangulated category $D^b(A)$. It will be denoted by $K_0(A)$. We will write $[M]$ for the class of an $A$-module $M$ in $K_0(A)$.

\subsection{Harish-Chandra induction and restriction}\label{sec:hcindres}
Given an $F$-stable parabolic subgroup $\mathbf{P}$ of $\mathbf{G}$ with Levi decomposition $\mathbf{P} = \mathbf{L} \mathbf{U}$, where $\mathbf{L}$ is $F$-stable, we denote by
\begin{gather*} \xymatrix{ \Lambda L\nmod \ar@/^1pc/[r]^{\HC_L^G} & \ar@/^1pc/[l]^{\rHC_L^G} \Lambda G\nmod }\end{gather*}
the Harish-Chandra induction and restriction functors. Under the assumption on~$\ell$, they form a biajdoint pair of exact functors. The natural transformations given by the adjunction (unit and counit) are denoted as follows
\begin{gather*} \operatorname{Id} \xrightarrow{\overline{\eta}^{L\subset G}} \HC_L^G \rHC_L^G \xrightarrow{\underline{\varepsilon}^{L\subset G}} \operatorname{Id} \qquad \text{and} \qquad \operatorname{Id}\xrightarrow{\underline{\eta}^{L\subset G}} \rHC_L^G \HC_L^G \xrightarrow{\overline{\varepsilon}^{L\subset G}} \operatorname{Id},\end{gather*}
where $\operatorname{Id}$ denotes here the identity functor. When it is clear from the context, we will usually drop the superscript $L \subset G$.

Let $\mathbf{Q}$ be another $F$-stable parabolic subgroup with $F$-stable Levi decomposition $\mathbf{Q} = \mathbf{M}\mathbf{V}$. If $\mathbf{P}\subset \mathbf{Q}$ and $\mathbf{L} \subset \mathbf{M}$ then the Harish-Chandra induction and restriction functors satisfy $\HC_M^G \circ \HC_L^M \simeq \HC_L^G$ and $\rHC_L^M \circ \rHC_M^G \simeq \rHC_L^G$. Together with the counit $\underline{\varepsilon}^{L \subset M}\colon \HC_L^M \rHC_L^M \longrightarrow 1$, this gives a natural transformation
\begin{gather*}\varphi^{L \subset M \subset G} \colon \ \HC_L^G \rHC_L^G \longrightarrow \HC_M^G \rHC_M^G.\end{gather*}
If $\mathbf{R}$ is any other $F$-stable parabolic subgroup containing $\mathbf{Q}$ with an $F$-stable Levi complement $\mathbf{N}$ containing $\mathbf{M}$ then the natural isomorphism giving the transitivity of Harish-Chandra induction and restriction can be chosen so that $\varphi^{M\subset N \subset G} \circ \varphi^{L \subset M \subset G} = \varphi^{L \subset N \subset G}$ (see for example \cite[Section~4]{CaEn} or \cite[Section~III.7.2]{DS}).

\subsection{Alvis--Curtis duality functor}\label{sec:acduality}
We now f\/ix a Borel subgroup $\mathbf{B}$ of $\mathbf{G}$ containing a maximal torus $\mathbf{T}$, both of which are assumed to be $F$-stable. Let $\Delta$ be the set of simple roots def\/ined by $\mathbf{B}$. The $F$-stable parabolic subgroups containing $\mathbf{B}$ are parametrized by $F$-stable subsets of $\Delta$. They have a unique Levi complement containing $\mathbf{T}$. Such Levi subgroups and parabolic subgroups are called \emph{standard}. Given $r \geq 0$, we denote by $\mathcal{L}_r$ the (f\/inite) set of $F$-stable standard Levi subgroups corresponding to a subset $I \subset \Delta$ satisfying $|I/F| = |\Delta/F|-r$. In particular $\mathcal{L}_0 = \{\mathbf{G}\}$ and $\mathcal{L}_{|\Delta/F|} = \{\mathbf{T}\}$. Following \cite{CaRi,DL82,DS} we can form the complex of exact functors
\begin{gather*} 0 \longrightarrow \HC_{T}^G \rHC_T^G \longrightarrow \cdots \longrightarrow
\bigoplus_{\mathbf{L} \in \mathcal{L}_2} \HC_L^G \rHC_L^G \longrightarrow
\bigoplus_{\mathbf{L} \in \mathcal{L}_1} \HC_L^G \rHC_L^G \longrightarrow \operatorname{Id} \longrightarrow 0,\end{gather*}
where $\operatorname{Id}$ is in degree $0$. It yields a functor $\mathsf{D}_G$ on the bounded derived category $D^b(\Lambda G)$ of f\/initely generated $\Lambda G$-modules, called the \emph{Alvis--Curtis duality functor}. Note that the original def\/inition of the duality by Alvis and Curtis \cite{Alv79,Cur80} refers to the linear endomorphism on $K_0(K G)$ induced by this functor. The complex above was introduced by Deligne--Lusztig in~\cite{DL82}.

\begin{Theorem}[Cabanes--Rickard \cite{CaRi}]
The functor $\mathsf{D}_G$ is a self-equivalence of $D^b(\Lambda G)$ satisfying
\begin{gather}\label{eq:curtistype}
\mathsf{D}_G \circ \mathsf{R}_L^G \simeq \mathsf{R}_L^G \circ \mathsf{D}_L[r]
\end{gather}
for every $\mathbf{L} \in \mathcal{L}_r$.
\end{Theorem}

Note that any quasi-inverse of $\mathsf{D}_G$ will also satisfy the relation~\eqref{eq:curtistype}, up to replacing $r$ by $-r$.

Chuang--Rouquier deduced in \cite{ChRou} from~\eqref{eq:curtistype} that the equivalence induced by $\mathsf{D}_G$ is \emph{perverse} with respect to the cuspidal depth (see below for the def\/inition). For the reader's convenience we recall here their argument.

Assume that $\Lambda$ is a f\/ield. The simple $\Lambda G$-modules are partitioned into Harish-Chandra series, see~\cite{Hiss}. Given a simple $\Lambda G$-module $S$, there exists $\mathbf{L}$ in $\mathcal{L}_s$ and a simple cuspidal $\Lambda L$-module~$X$ such that~$S$ appears in the head (or equivalently in the socle) of $\mathsf{R}_L^G(X)$. The pair $(\mathbf{L},X)$ is unique up to $G$-conjugation. We say that $S$ lies in the Harish-Chandra series of the cuspidal pair $(\mathbf{L},X)$ and we call $s$ the \emph{cuspidal depth} of $S$. In particular, the cuspidal modules are the modules with cuspidal depth zero.

\begin{Proposition}[Chuang--Rouquier \cite{ChRou}]\label{prop:perverse}
Assume that $\Lambda$ is a field. Let $S$ be a simple $\Lambda G$-module of cuspidal depth $s$. Then
\begin{itemize}\itemsep=0pt
 \item[$(i)$] $H^i(\mathsf{D}_G(S)) = 0$ for $i > 0$ and $i < -s$.
 \item[$(ii)$] The composition factors of $H^i(\mathsf{D}_G(S))$ have depth at most $s$.
 \item[$(iii)$] Among all the composition factors of $\bigoplus_i H^{i}(\mathsf{D}_G(S))$, there is a unique composition factor of depth~$s$. It is a submodule of $H^{-s}(\mathsf{D}_G(S))$, and it lies in the same Harish-Chandra series as~$S$.
 \end{itemize}
\end{Proposition}

\begin{proof}We denote by $(L,X)$ a cuspidal pair associated with $S$. Given any other cuspidal pair $(M, Y)$ with $\mathbf{M} \in \mathcal{L}_r$ and $n \in \mathbb{Z}$ we have
\begin{gather}
\operatorname{Hom}_{D^b(\Lambda G)} \big(\mathsf{R}_M^G(Y),\mathsf{D}_G(S)[n]\big)
 \simeq \operatorname{Hom}_{D^b(\Lambda G)} \big(\mathsf{D}_G^{-1}(\mathsf{R}_M^G(Y)),S[n]\big) \nonumber\\
\hphantom{\operatorname{Hom}_{D^b(\Lambda G)} \big(\mathsf{R}_M^G(Y),\mathsf{D}_G(S)[n]\big)}{} \simeq \operatorname{Hom}_{D^b(\Lambda G)} \big(\mathsf{R}_M^G(\mathsf{D}_M^{-1}(Y))[-r],S[n]\big) \nonumber\\
\hphantom{\operatorname{Hom}_{D^b(\Lambda G)} \big(\mathsf{R}_M^G(Y),\mathsf{D}_G(S)[n]\big)}{} \simeq \operatorname{Hom}_{D^b(\Lambda M)} \big(Y,{}^*\mathsf{R}_M^G(S)[r+n]\big),\label{eq:proofpervers}
\end{gather}
where in the last equality we used that $\mathsf{D}_M^{-1}(Y) \simeq Y$ since $Y$ is cuspidal. In particular, it is zero when $r<-n $ or when $M$ does not contain a $G$-conjugate of $L$, so in particular when $r > s$.

Take $n$ to be the smallest integer such that $H^n(\mathsf{D}_G(S)) \neq 0$ and consider a simple $\Lambda G$-modu\-le~$T$ in the socle of $H^n(\mathsf{D}_G(S))$. Let $(M,Y)$ be the cuspidal pair above which $T$ lies. Then the composition $\mathsf{R}_M^G(Y) \twoheadrightarrow T \hookrightarrow H^n(\mathsf{D}_G(S))$ yields a non-zero element in \begin{gather*} \operatorname{Hom}_{D^b(\Lambda G)} \big(\mathsf{R}_M^G(Y),\mathsf{D}_G(S)[n]\big).\end{gather*}
From~\eqref{eq:proofpervers} we must have $-n \leq r \leq s$ which proves~(i). Furthermore, if $n = -s$ then $r$ and $s$ are equal, and in that case
\begin{gather*}
\operatorname{Hom}_{D^b(\Lambda G)} \big(\mathsf{R}_M^G(Y),\mathsf{D}_G(S)[-s]\big)
\simeq \operatorname{Hom}_{D^b(\Lambda M)} \big(Y,{}^*\mathsf{R}_M^G(S)\big)
 \simeq \operatorname{Hom}_{\Lambda M} \big(Y,{}^*\mathsf{R}_M^G(S)\big).
\end{gather*}
By the Mackey formula, ${}^*\mathsf{R}_M^G(S)$ is isomorphic to a direct sum of $G$-conjugates of $X$. Therefore if $T$ lies in the socle of $H^{-s}(\mathsf{D}_G(S))$ then $Y$ is $G$-conjugate to $X$ which means that $T$ and $S$ lie in the same Harish-Chandra series.

Now if we replace $Y$ by its projective cover $P_Y$ in \eqref{eq:proofpervers} we get
\begin{gather*}\operatorname{Hom}_{D^b(\Lambda G)} \big(\mathsf{R}_M^G(P_Y),\mathsf{D}_G(S)[n]\big)
\simeq \operatorname{Hom}_{D^b(\Lambda M)} \big(P_Y,D_M\big({}^*\mathsf{R}_M^G(S)\big)[r+n]\big),\end{gather*}
which again is zero unless $ r \leq s$ or unless $r=s$ and $(M,Y)$ is conjugate to $(L,X)$. In that latter case we have
\begin{gather*}
\operatorname{Hom}_{D^b(\Lambda G)} \big(\mathsf{R}_L^G(P_X),\mathsf{D}_G(S)[n]\big) \simeq \operatorname{Hom}_{D^b(\Lambda L)} \big(P_X,\mathsf{D}_L\big({}^*\mathsf{R}_L^G(S)\big)[n+s]\big) \\
\hphantom{\operatorname{Hom}_{D^b(\Lambda G)} \big(\mathsf{R}_L^G(P_X),\mathsf{D}_G(S)[n]\big)}{}
 \simeq \operatorname{Hom}_{D^b(\Lambda L)} \big(P_X,{}^*\mathsf{R}_L^G(S)[n+s]\big),
\end{gather*}
since ${}^*\mathsf{R}_L^G(S)$ is a sum of conjugates of $X$. Therefore the composition factors of $D_G(S)$ lying in the Harish-Chandra series of $(L,X)$ can only appear in degree $-s$. They appear with multiplicity one since $\mathsf{D}_G$ is a self-equivalence, which proves~(iii).
\end{proof}

Using the property that $\mathsf{D}_G$ is a perverse equivalence (given in Proposition \ref{prop:perverse}) we can def\/ine a bijection on the set of simple $\Lambda G$-modules as follows: given a simple $\Lambda G$-module $S$ with cuspidal depth $s$ we def\/ine $\mathsf{d}_G(S)$ to be the unique simple $\Lambda G$-module with depth $s$ which occurs as a composition factor in the cohomology of $\mathsf{D}_G(S)$. Note that $\mathsf{d}_G(S)$ and $S$ lie in the same Harish-Chandra series.

\subsection{Compatibility with Hecke algebras}\label{sec:comphecke}
Given a cuspidal pair $(L,X)$ of $G$, we can form the endomorphism algebra
\begin{gather*} \mathcal{H}_G(L,X)= \operatorname{End}_G\big(\HC_L^G(X)\big).\end{gather*} By \cite[Theorem~2.4]{GHM}, the isomorphism classes of simple quotients of $\HC_L^G(X)$ (the Harish-Chandra series of $(L,X)$) are parametrized by the simple representations of $\mathcal{H}_G(L,X)$. The structure of this algebra was studied for example in \cite[Section~3]{GHM}; it is in general very close to be a~Iwahori--Hecke algebra of a Coxeter group.

Let $\mathbf{M}$ be an $F$-stable standard Levi subgroup of $\mathbf{G}$ containing $\mathbf{L}$. Since $\HC_M^G$ is fully-faithful, $\mathcal{H}_M(L,X) = \operatorname{End}_G(\HC_L^M(X))$ embedds naturally as a subalgebra of $\mathcal{H}_G(L,X)$. To this embedding one can associate the induction and restriction functors
\begin{gather*} \xymatrix{ \modn\mathcal{H}_M(L,X) \ar@/^1pc/[r]^{\operatorname{Ind}_{\mathcal{H}_M}^{\mathcal{H}_G}} & \ar@/^1pc/[l]^{\operatorname{Res}_{\mathcal{H}_M}^{\mathcal{H}_G}} \modn\mathcal{H}_G(L,X) }\end{gather*}
between the categories of right modules. The purpose of this section is to compare the Alvis--Curtis duality functor $\mathsf{D}_G$ for the group with a similar functor $\mathsf{D}_\mathcal{H}$ of the Hecke algebra. From now on we shall f\/ix the cuspidal pair $(L,X)$, and we will denote simply $\mathcal{H}_G$ and $\mathcal{H}_M$ the endomorphism algebras of $\HC_L^G(X)$ and $\HC_L^M(X)$ respectively.

Let $Y$ be a (non-necessarily cuspidal) $\Lambda M$-module. We consider the natural transforma\-tion~$\Theta_{M,Y}$ def\/ined so that the following diagram commutes:
\begin{gather*}\xymatrix{ \operatorname{Hom}_M(Y,\rHC_M^G(-))\otimes_{\operatorname{End}_M(Y)} \operatorname{Hom}_M(\rHC_M^G \HC_M^G (Y),Y) \ar[r]^{\qquad \qquad \ \ \text{mult}} \ar[d]^\sim& \operatorname{Hom}_M(\rHC_M^G \HC_M^G(Y),\rHC_M^G(-)) \ar[d]^\sim \\
 \operatorname{Hom}_G(\HC_M^G(Y),-) \otimes_{\operatorname{End}_M(Y)} \operatorname{End}_G(\HC_M^G(Y))
 \ar[r]^{\quad \qquad \Theta_{M,Y}} & \operatorname{Hom}_G(\HC_M^G(Y), \HC_M^G \rHC_M^G(-)).}\end{gather*}
In other words, $\Theta_{M,Y}$ is def\/ined on the objects by
\begin{gather*}\Theta_{M,Y}(f\otimes h) = \HC_M^G \rHC_M^G(f) \circ \HC_M^G(\underline{\eta}_{Y})\circ h.\end{gather*}
Using this description one can check that $\Theta_{M,Y}$ is compatible with the right action of \linebreak $\operatorname{End}_G(\HC_M^G(Y))$. Therefore it is a well-def\/ined natural transformation between functors from $\Lambda G\nmod$ to $\modn\operatorname{End}_G(\HC_M^G(Y))$.

Let $\mathbf{N}$ be another standard $F$-stable Levi subgroup of $\mathbf{G}$ with $\mathbf{M} \subset \mathbf{N}$. Recall from Section~\ref{sec:hcindres} the natural transformation $\varphi^{M\subset N \subset G}\colon \HC_M^G \rHC_M^G \longrightarrow \HC_N^G \rHC_N^G$ which was needed for the construction of $\mathsf{D}_G$. By composition it induces a natural transformation
\begin{gather*}\operatorname{Hom}_G\big(\HC_M^G(Y), \HC_M^G \rHC_M^G(-)\big)
\longrightarrow \operatorname{Hom}_G\big(\HC_M^G(Y), \HC_N^G \rHC_N^G(-)\big).\end{gather*}

\begin{Proposition}\label{prop:compatibilityarrows}
The following diagram is commutative
$$\text{\footnotesize $\xymatrix@!C=8cm{ \operatorname{Hom}_G\big(\HC_M^G(Y),-\big) \otimes_{\operatorname{End}_M(Y)}
 \operatorname{End}_G\big(\HC_M^G(Y)\big)
 \ar[r]^{\qquad\quad\Theta_{M,Y}} \ar[d] & \operatorname{Hom}_G\big(\HC_M^G(Y), \HC_M^G \rHC_M^G(-)\big) \ar[d] \\
\operatorname{Hom}_G\big(\HC_M^G(Y),-\big) \otimes_{\operatorname{End}_N\big(\HC_M^N(Y)\big)} \operatorname{End}_G\big(\HC_M^G(Y)\big)
 \ar[d]^{\sim} & \operatorname{Hom}_G\big(\HC_M^G(Y), \HC_N^G \rHC_N^G(-)\big) \ar[d]^\sim\\
 \operatorname{Hom}_G\big(\HC_N^G(\HC_M^N(Y)),-\big) \otimes_{\operatorname{End}_N\big(\HC_M^N(Y)\big)}\operatorname{End}_G\big(\HC_N^G\big(\HC_M^N(Y)\big)\big)
 \ar[r]^{\qquad\quad \qquad \Theta_{N,\HC_M^N(Y)}} & \operatorname{Hom}_G\big(\HC_N^G\big(\HC_M^N(Y)\big), \HC_N^G \rHC_N^G(-)\big).}$
 }$$
\end{Proposition}

\begin{proof} Since $\HC_M^N$ is faithful one can see $\operatorname{End}_M(Y)$ as a subalgebra of $\operatorname{End}_N(\HC_M^N(Y))$. Then the f\/irst vertical map on the left-hand side of the diagram is the canonical projection. We will write $t\colon \HC_M^G\mathop{\longrightarrow}\limits^\sim \HC_N^G \HC_M^N $ (resp.\ $t^*\colon \rHC_M^G\mathop{\longrightarrow}\limits^\sim \rHC_M^N \rHC_N^G $) for the isomorphism of functors coming from the transitivity of Harish-Chandra induction (resp.\ restriction).

Let $Z$ be a $\Lambda G$-module, $h \in \operatorname{End}(\HC_M^G(Y))$ and $f \in \operatorname{Hom}_G(\HC_M^G(Y),Z)$. The commutativity of the diagram is equivalent to the relation
\begin{gather*}
 \HC_N^G \rHC_N^G\big(f \circ t_Y^{-1}\big) \circ \HC_N^G\big(\big(\underline{\eta}^{N\subset G}\big)_{\HC_M^N(Y)}\big)\circ \big( t_Y \circ h \circ t_Y^{-1}\big) \\
 \qquad{} = \big(\varphi^{M\subset N \subset G}\big)_{Z} \circ \HC_M^G \rHC_M^G(f) \circ \HC_M^G\big(\big(\underline{\eta}^{M\subset G}\big)_{Y}\big)\circ h \circ t_Y^{-1}. \end{gather*}
Since $\big(\varphi^{M\subset N \subset G}\big)_{Z} \circ \HC_M^G \rHC_M^G(f) = \HC_N^G \rHC_N^G(f) \circ \big(\varphi^{M\subset N \subset G}\big)_{\HC_M^G(Y)}$ it is enough to show that
\begin{gather}\label{eq:commute}
 \big(\varphi^{M\subset N \subset G}\big)_{\HC_M^G(Y)} \circ \HC_M^G\big(\big(\underline{\eta}^{M\subset G}\big)_{Y}\big) = \HC_N^G\rHC_N^G\big(t_Y^{-1}\big) \circ \HC_N^G\big(\big(\underline{\eta}^{N\subset G}\big)_{\HC_M^N(Y)}\big)\circ t_Y.
 \end{gather}
This comes from the following commutative diagram
\begin{gather*}\xymatrix{
\HC_M^G \rHC_M^G \HC_M^G \ar[r]^t \ar@/^2pc/[rrr]^{\varphi^{M\subset N \subset G}}& \HC_N^G \HC_M^N \rHC_M^G \HC_M^G
\ar[r]^{t^*} & \HC_N^G \HC_M^N \rHC_M^N \rHC_N^G \HC_M^G \ar[r]^{\hskip 5mm \underline{\varepsilon}^{M\subset N}} \ar[d]^t & \fbox{$\HC_N^G \rHC_N^G \HC_M^G$} \\
\fbox{$\HC_M^G$} \ar[u]_{\underline{\eta}^{M\subset G}} \ar[r]^t& \HC_N^G\HC_M^N \ar@/_4pc/@{=}[rrd] \ar[u]_{\underline{\eta}^{M\subset G}} \ar[rd]_{\underline{\eta}^{M\subset N}} &\HC_N^G \HC_M^N \rHC_M^N \rHC_N^G \HC_N^G
\HC_M^N \ar[r]^{\hskip 8mm \underline{\varepsilon}^{M\subset N}} & \ar[u]_{t^{-1}} \HC_N^G \rHC_N^G \HC_N^G \HC_M^N \\
& & \HC_N^G \HC_M^N \rHC_M^N \HC_M^N \ar[r]_{\underline{\varepsilon}^{M\subset N}} \ar[u]_{\underline{\eta}^{N\subset G}} & \HC_N^G \HC_M^N. \ar[u]_{\underline{\eta}^{N\subset G}}}
 \end{gather*}

\vspace{5mm}

\noindent
where for simplicity we did not write the identity natural transformations. The only non-trivial commutative subdiagram is the central one, which comes from the relation
\begin{gather*} \big(\rHC_M^N \cdot \underline{\eta}^{N\subset G}\cdot \HC_M^N\big) \circ \underline{\eta}^{M\subset N} = (t^* \cdot t) \circ \underline{\eta}^{N\subset G},\end{gather*}
which we assume to hold by our choice of $t$ and $t^*$. For more on the compatibility of the unit, counit, $t$ and $t^*$ see \cite[Part~3]{DS}. Now the relation~\eqref{eq:commute} comes from the equality between two natural transformations between the functors $\HC_M^G$ and $\HC_N^G \rHC_N^G \HC_M^G$, given by the top and the bottom arrows respectively.
\end{proof}

In the particular case where $Y = \HC_L^G(X)$ with $(L,X)$ being a cuspidal pair, the natural transformation $\Theta_{M,\HC_L^M(X)}$ becomes
\begin{gather*} \operatorname{Ind}_{\mathcal{H}_M}^{\mathcal{H}_G} \operatorname{Res}_{\mathcal{H}_M}^{\mathcal{H}_G}\big( \operatorname{Hom}_G(\HC_L^G(X), - )\big) \xrightarrow{\Theta_{M,\HC_L^M(X)}} \operatorname{Hom}_G(\HC_L^G(X), \HC_M^G \rHC_M^G(-)). \end{gather*}
It can be seen as a way to intertwine the endofunctor $\operatorname{Ind}_{\mathcal{H}_M}^{\mathcal{H}_G} \operatorname{Res}_{\mathcal{H}_M}^{\mathcal{H}_G}$ of $\modn\mathcal{H}_G$ and the endofunctor $\HC_M^G \rHC_M^G$ of $\Lambda G\nmod$ via the functor $\operatorname{Hom}_G(\HC_L^G(X), - )$, as shown in the following diagram
\begin{gather*}\xymatrix@!C=1.5cm{kG\nmod \ar[rr]^{\operatorname{Hom}_G(\HC_L^G(X), - )} \ar[dd]_{\HC_M^G \rHC_M^G}
 & & \modn\mathcal{H}_G \ar[dd]^{\operatorname{Ind}_{\mathcal{H}_M}^{\mathcal{H}_G} \operatorname{Res}_{\mathcal{H}_M}^{\mathcal{H}_G}} \\
 & \rotatebox{35}{$\Longleftarrow$} \\
 kG\nmod \ar[rr]^{\operatorname{Hom}_G(\HC_L^G(X), - )} & & \modn\mathcal{H}_G.}\end{gather*}
In this diagram the central double arrow represents the natural transformation $\Theta_{M,\HC_L^M(X)}$. We give a condition for this transformation to be an isomorphism. Note that similar results were obtained by Dipper--Du in~\cite[Theorem~1.3.2]{DipDu} in the case of general linear groups and by Seeber~\cite{See} with coinduction instead of induction.

\begin{Proposition}\label{prop:compatibilityindred}
Assume that $\Lambda$ is a field. Let $(L,X)$ be a cuspidal pair of $G$, and $\mathbf{M}$ be a~standard Levi of $\mathbf{G}$ containing $\mathbf{L}$. Assume that any cuspidal pair of $M$ which is $G$-conjugate to $(L,X)$ is actually $M$-conjugate to $(L,X)$. Then $\Theta_{M,\HC_L^M(X)}$ is an isomorphism.
\end{Proposition}

\begin{proof}
By def\/inition of $\Theta_{M,\HC_L^M(X)}$ it is enough to show that the multiplication map
\begin{gather*}
\xymatrix@!C=1.5cm{\operatorname{Hom}_M\big(\HC_L^M(X),\rHC_M^G(-)\big) \otimes_{\operatorname{End}_M\big(\HC_L^M(X)\big)} \operatorname{Hom}_M\big(\rHC_M^G \HC_L^G (X),\HC_L^M(X)\big) \ar[d] \\ \operatorname{Hom}_M\big(\rHC_M^G \HC_L^G(X),\rHC_M^G(-)\big)}
\end{gather*}
is an isomorphism of $k$-vector spaces. Now by the Mackey formula \cite[Proposition~1.5]{CaEn} we have
\begin{gather*} \rHC_M^G \HC_L^G (X) \simeq \bigoplus_{\begin{subarray}{c} x \in Q\backslash G/P \\ {}^x \mathbf{L} \subset \mathbf{M} \end{subarray}} \HC_{{}^x L}^M(X^x).\end{gather*}
Under the assumption on $(L,X)$, any cuspidal pair $({}^xL,X^x)$ with ${}^xL \subset M$ is conjugate to $(L,X)$ under $M$. In particular $\HC_{{}^x L}^M(X^x) \simeq \HC_L^M(X)$ and we deduce that each of the composition maps
\begin{gather*}\xymatrix@!C=1.5cm{\operatorname{Hom}_M\big(\HC_L^M(X),\rHC_M^G(-)\big) \otimes_{\operatorname{End}_M\big(\HC_L^M(X)\big)} \operatorname{Hom}_M\big(\HC_{{}^x L}^M(X^x),\HC_L^M(X)\big)
 \ar[d] \\ \operatorname{Hom}_M\big(\HC_{{}^x L}^M(X^x),\rHC_M^G(-)\big)}\end{gather*}
is an isomorphism.
\end{proof}

As in the case of the f\/inite group $G$ we can form a complex of functors coming from induction and restriction in $\mathcal{H}_G$. Given $r \geq 0$, we denote by $\mathcal{L}_r(\mathbf{L})$ the subset of $\mathcal{L}_r$ of standard Levi subgroups containing $\mathbf{L}$. The complex of functors
\begin{gather*} 0 \longrightarrow \operatorname{Ind}_{\mathcal{H}_L}^{\mathcal{H}_G} \operatorname{Res}_{\mathcal{H}_L}^{\mathcal{H}_G} \longrightarrow \cdots \longrightarrow
\bigoplus_{\mathbf{M} \in \mathcal{L}_1(\mathbf{L})}\operatorname{Ind}_{\mathcal{H}_M}^{\mathcal{H}_G} \operatorname{Res}_{\mathcal{H}_M}^{\mathcal{H}_G} \longrightarrow \operatorname{Id} \longrightarrow 0,\end{gather*}
where $\operatorname{Id}$ is in degree $0$, induces a triangulated functor $\mathsf{D}_{\mathcal{H}_G}$ in $D^b(\modn\mathcal{H}_G)$ whenever each term of the complex is exact. For that property to hold we need to assume that $\mathcal{H}_G$ is f\/lat over each subalgebra of the form $\mathcal{H}_M$. Combining Propositions \ref{prop:compatibilityarrows} and \ref{prop:compatibilityindred} we get

\begin{Theorem}\label{thm:main}
Assume that $\Lambda$ is a field. Let $\mathbf{L}$ be an $F$-stable standard Levi subgroup of $\mathbf{G}$ and~$X$ be a cuspidal $\Lambda L$-module. Let $\mathcal{H}_G = \operatorname{End}_G(\HC_L^G(X))$. Assume that for every $F$-stable standard Levi subgroup $\mathbf{M}$ of $\mathbf{G}$ containing $\mathbf{L}$ we have:
\begin{itemize}\itemsep=0pt
\item[$(i)$] $\mathcal{H}_G$ is flat over $\mathcal{H}_M = \operatorname{End}_M\big(\HC_L^M(X)\big)$.
\item[$(ii)$] Every cuspidal pair of $M$ which is $G$-conjugate to $(L,X)$ is actually $M$-conjugate to $(L,X)$.
\end{itemize}
Then there is a natural isomorphism of endofunctors of $D^b(\modn\mathcal{H}_G)$
\begin{gather*}\mathsf{D}_{\mathcal{H}_G}\big(\mathrm{RHom}_G\big(\HC_L^G(X),-\big)\big) \simto \mathrm{RHom}_G\big(\HC_L^G(X),\mathsf{D}_G(-)\big). \end{gather*}
\end{Theorem}

\section[Unipotent representations of $\mathrm{GL}_n(q)$]{Unipotent representations of $\boldsymbol{\mathrm{GL}_n(q)}$}\label{section3}

In this section we show how to use Theorem \ref{thm:main} to compute $d_{\mathrm{GL}_n(q)}(S)$ for every unipotent simple $k\mathrm{GL}_n(q)$-module $S$. This will involve an involution on the set of partitions of $n$ generalizing the Mullineux involution~\cite{Mu}. So here $\mathbf{G} = \mathrm{GL}_n(\overline{\mathbb{F}}_p)$ will be the general linear group over an algebraic closure of $\mathbb{F}_p$, and $F\colon \mathbf{G} \longrightarrow \mathbf{G}$ the standard Frobenius endomorphism, raising the entries of a matrix to the $q$-th power.

\subsection{Partitions}\label{sec:partitions}
A \emph{partition} $\lambda$ of $n\in \mathbb{N}$ is a non-increasing sequence $(\lambda_1 \geq \lambda_2 \geq \cdots \geq \lambda_r)$ of positive integers which add up to $n$. By convention, $\varnothing$ is the unique partition of $0$ and is called the {\it empty partition}. The set of partitions of $n$
will be denoted by $\mathcal{P}(n)$, and the set of all partitions by $\mathcal{P} := \sqcup_{n \in \mathbb{N}} \mathcal{P}(n)$. We shall also use the notation $\Lambda = (1^{r_1},2^{r2}, \ldots, n^{r_n})$ where $r_i$ denotes the multiplicity of $i$ in the sequence $\lambda$.

Given $\lambda$ and $\mu$ two partitions of $n_1\in \mathbb{N}$ and $n_2\in \mathbb{N}$ respectively, we denote by $\lambda \sqcup \mu$ the partition of $n_1+n_2$ obtained by concatenation of the two partitions and by reordering the parts to obtain a partition. If $\lambda$ is a partition of $n$ and $k\in \mathbb{N}$, we denote by $\lambda^k$ the partition $\underbrace{\lambda \sqcup\lambda \sqcup \cdots\sqcup \lambda}_{k\text{ times}}$.

Let $d\in \mathbb{N}_{>1}$. A partition $\lambda$ is called {\it $d$-regular} if no part in $\lambda$ is repeated $d$ or more times. Each partition $\lambda$ can be decomposed uniquely as $\lambda=\mu\sqcup \nu^d$ where $\mu$ is $d$-regular. Then $\lambda$ is $d$-regular if and only if $\nu$ is empty. The set of $d$-regular partitions of $n$ is denoted by $\operatorname{Reg}_d (n)$. This set has a remarkable involution $\mathsf{M}_d$ called the \emph{Mullineux involution} which will be def\/ined in the next section (see Section~\ref{sec:crystalmull} for its interpretation in terms of crystals).

More generally we shall decompose partitions with respect to two integers. Recall that $\ell$ is a~prime number. If $\lambda = (1^{r_1},2^{r2}, \ldots, n^{r_n})$ is a partition of $n$, we can decompose the integers $r_i$ as
 \begin{gather*} r_i = r_{i,-1} + d r_{i,0} + d\ell r_{i,1} + d\ell^2 r_{i,2} + \cdots + d\ell^n r_{i,n} \end{gather*}
 with $0 \leq r_{i,-1} < d$ and $0 \leq r_{i,j} < \ell$ for all $j \geq 0$. If we def\/ine the partition $\lambda_{(j)} = (1^{r_{1,j}}, 2^{r_{2,j}}, \ldots, n^{r_{n,j}})$, then
\begin{gather*}\lambda = \lambda_{(-1)} \sqcup (\lambda_{(0)})^d \sqcup (\lambda_{(1)})^{d\ell} \sqcup \cdots \sqcup (\lambda_{(n)})^{d\ell^n},\end{gather*}
where $\lambda_{(-1)}$ is $d$-regular and $\lambda_{(j)}$ is $\ell$-regular for all $j \geq 0$. This decomposition is called the \emph{$d$-$\ell$-adic decomposition of $\lambda$}~\cite{DipDu}.

\subsection[Hecke algebras of type $A$ and the Mullineux involution]{Hecke algebras of type $\boldsymbol{A}$ and the Mullineux involution}\label{sec:hecke}
Let $q\in k^\times$ and $m\geq 1$. We denote by $\mathcal{H}_q (\mathfrak{S}_m)$ the Iwahori--Hecke algebra of the symmetric group $\mathfrak{S}_m$ over $k$, with parameter $q$. It has a $k$-basis $\{T_w\}_{w\in \mathfrak{S}_m}$ satisfying the following relations, for $w \in \mathfrak{S}_m$ and $s = (i,i+1)$ a simple ref\/lection:
\begin{gather*}T_w T_s= \begin{cases} T_{ws} & \text{if $\ell(ws)>\ell(w)$ (i.e., if $w(i) < w(i+1)$)}, \\
qT_{ws}+(q-1) T_w &\text{otherwise}.
\end{cases} \end{gather*}
In particular the basis elements corresponding to the simple ref\/lections generate $\mathcal{H}_q (\mathfrak{S}_m)$ as an algebra, and they satisfy the relation $(T_s-q)(T_s+1) = 0$.

Let us consider the integer
\begin{gather*}e:=\min \big\{i\geq 0\, |\, 1+q+q^2+\cdots+q^{i-1}=0\big\}\in \mathbb{N}_{>1}.\end{gather*}
It is equal to the order of $q$ in $k^\times$ when $q\neq 1$, and to $\ell = \mathrm{char}\,k$ when $q =1$.
Then the set of simple $\mathcal{H}_q (\mathfrak{S}_m)$-modules is parametrized by the set of $e$-regular partitions of $m$. Given an $e$-regular partition $\lambda$, we will denote by $D(\lambda)$ the corresponding simple module. Note that when $q=1$, the Hecke algebra $\mathcal{H}_q (\mathfrak{S}_m)$ coincides with the group algebra of $\mathfrak{S}_m$ over $k$, whose irreducible representations are parametrized by $\ell$-regular partitions of~$m$.

The map $\alpha\colon T_w \longmapsto (-q)^{\ell(w)} (T_{w^{-1}})^{-1}$ is an algebra automorphism of $\mathcal{H}_q (\mathfrak{S}_m)$ of order~$2$. This follows from the fact that $-qT_s^{-1}$ satisf\/ies the same quadratic equation as $T_s$. The automor\-phism~$\alpha$ induces a permutation $\alpha^*$ on the set of simple $\mathcal{H}_q (\mathfrak{S}_m)$-modules, and therefore on the set of $e$-regular partitions. In other words, the exists an involution $\mathsf{M}_e$ on $\operatorname{Reg}_e (n)$, called the \emph{Mullineux involution}, such that for any $e$-regular partition $\lambda$
\begin{gather*} \alpha^* (D(\lambda)) \simeq D(\mathsf{M}_e(\lambda)).\end{gather*}
In the case where $q=1$, the involution $\alpha$ is just the multiplication by the sign representation $\varepsilon$ in the group algebra of $\mathfrak{S}_m$ and $D(\mathsf{M}_\ell(\lambda)) \simeq D(\lambda) \otimes \varepsilon$.

The involution $\mathsf{M}_\ell$ was introduced by Mullineux, who suggested in~\cite{Mu} a conjectural explicit combinatorial algorithm to compute it. This conjecture was subsequently proved by Ford--Kleshchev in~\cite{ForKle}. An interpretation in terms of crystals was later given by Kleshchev~\cite{Kle} (for the case $q=1$) and by Brundan~\cite{Brun} for the general case. We will review their result in Section~\ref{sec:crystalmull}.

\begin{Remark}\label{rem:asymptotic}
When $e>m$, every partition of $m$ is $e$-regular. The Hecke algebra is actually semi-simple in that case and the Mullineux involution $\mathsf{M}_e$ corresponds to conjugating partitions.
\end{Remark}

Recall that $\mathfrak{S}_m$ has a structure of a Coxeter group, where the simple ref\/lections are given by the transpositions $(i,i+1)$. Given a parabolic subgroup~$\mathfrak{S}$ of~$\mathfrak{S}_m$, one can consider the subalgebra of~$\mathcal{H}_q (\mathfrak{S}_m)$ generated by $\{T_w\}_{w \in \mathfrak{S}}$. It corresponds to the Hecke algebra $\mathcal{H}_q (\mathfrak{S})$ of~$\mathfrak{S}$ with parameter~$q$, and $\mathcal{H}_q (\mathfrak{S}_m)$ is f\/lat as a module over that subalgebra. It is even free, with basis given by the elements $T_w$ where $w$ runs over a set of representatives of $\mathfrak{S}\backslash \mathfrak{S}_m$ with minimal length. Therefore, following Section~\ref{sec:comphecke} (see also~\cite{LinSch}) we can use the induction and restriction functors to def\/ine a duality functor $\mathsf{D}_\mathcal{H}$ on the bounded derived category of f\/initely generated $\mathcal{H}_q (\mathfrak{S}_m)$-modules. For Hecke algebras this duality functor is actually a shifted Morita equivalence.

\begin{Theorem}[Linckelmann--Schroll \cite{LinSch}]\label{thm:dualityhecke}
Given a finitely generated right $\mathcal{H}_q (\mathfrak{S}_m)$-module $X$ we have
\begin{gather*} \mathsf{D}_\mathcal{H} (X) \simeq \alpha^*(X) [-m+1]\end{gather*}
in $D^b(\modn\mathcal{H}_q (\mathfrak{S}_m))$.
\end{Theorem}

In particular when $X = D(\lambda)$ is simple we get $\mathsf{D}_\mathcal{H} (D(\lambda)) \simeq D(\mathsf{M}_e(\lambda)) [-m+1]$.

\subsection[Harish-Chandra series of $\mathrm{GL}_n(q)$ and a generalized Mullineux involution]{Harish-Chandra series of $\boldsymbol{\mathrm{GL}_n(q)}$ and a generalized Mullineux involution}\label{sec:hcgln}

From now on $\mathbf{G} = \mathrm{GL}_n(\overline{\mathbb{F}}_p)$ is the general linear group over an algebraic closure of $\mathbb{F}_p$, and $F\colon \mathbf{G} \longrightarrow \mathbf{G}$ the standard Frobenius endomorphism, raising the entries of a matrix to the $q$-th power. The corresponding f\/inite reductive group is $G = \mathrm{GL}_n(q)$.

Recall that the unipotent simple $kG$-modules are parametrized by partitions of $n$. We will denote by $S(\lambda)$ the simple $kG$-module corresponding to the partition $\lambda$. We review now the results in \cite{DipDu,GHM94} on the partition of the unipotent representations into Harish-Chandra series (see also \cite[Section~19]{CaEn}). This classif\/ication depends on both $\ell$, and the integer $e > 1$ def\/ined as the order of $q$ modulo $\ell$ (with the convention that $e = \ell$ if $q \equiv 1$ modulo $\ell$).

By \cite[Corollary~4.3.13]{DipDu} $S(\lambda)$ is cuspidal if and only if $\lambda = 1$ or $\lambda = 1^{e\ell^i}$ for some $i \geq 0$. In particular, given $n \geq 1$ there is at most one cuspidal unipotent simple $k\mathrm{GL}_n(q)$-module, and there is exactly one if and only if $n =1$ or $n= e\ell^i$ for some $i \geq 0$. By considering products of such representations one can construct any cuspidal pair of $G$. To this end, we introduce the set
\begin{gather*}\mathcal{N}(n) = \big\{ {\bf m} = (m_{-1},m_0,\ldots,m_n) \, | \, n = m_{-1} + em_0 + e\ell m_1+\cdots+e\ell^nm_n\big\}.\end{gather*}
Given ${\bf m} \in \mathcal{N}(n)$ we def\/ine the standard Levi subgroup
\begin{gather*}L_\mathbf{m} = \mathrm{GL}_{1}(q)^{m_{-1}} \times \mathrm{GL}_{e}(q)^{m_0} \times \cdots \times \mathrm{GL}_{e\ell^n}(q)^{m_n}.\end{gather*}
It has a unique cuspidal unipotent simple module $X_{\mathbf{m}}$, and all the cuspidal pairs $(L_{\mathbf{m}},X_{\mathbf{m}})$ of~$G$ are obtained this way for various ${\mathbf{m}} \in \mathcal{N}(n)$.

The simple $kG$-modules lying in the corresponding Harish-Chandra series are parametrized by the irreducible representations of the endomorphism algebra $\mathcal{H}_\mathbf{m} = \operatorname{End}_{G}\big(\HC_{L_\mathbf{m}}^G(X_\mathbf{m})\big)$. By \cite[Section~3.4]{DipDu} (see also \cite[Lemma~19.24]{CaEn}), there is a natural isomorphism of algebras
\begin{gather}\label{eq:isohecke}
\mathcal{H}_\mathbf{m} \simeq \mathcal{H}_{q}(\mathfrak{S}_{m_{-1}}) \otimes k \mathfrak{S}_{m_0} \otimes \cdots \otimes k \mathfrak{S}_{m_n},
\end{gather}
where $\mathcal{H}_{q}(\mathfrak{S}_{m_{-1}})$ is the Iwahori--Hecke algebra of $\mathfrak{S}_{m_{-1}}$ introduced in the previous section. Therefore the simple modules of $\mathcal{H}_\mathbf{m}$ are parametrized by tuples of partitions
\begin{gather*} \boldsymbol{\lambda}= (\lambda_{(-1)},\lambda_{(0)},\ldots,\lambda_{(n)}),\end{gather*} where $\lambda_{(-1)}$ is an $e$-regular partition of $m_{-1}$ and each $\lambda_{(i)}$ for $i \geq 0$ is an $\ell$-regular partition of~$m_{i}$.

By the tensor product theorem of Dipper--Du \cite[Corollary~4.3.11]{DipDu} (see also \cite[Theorem~19.20]{CaEn}), the simple module $S(\lambda)$ of $G$ attached to this multipartition $\boldsymbol\lambda$ is given by
\begin{gather}\label{eq:multipart}
\lambda = \lambda_{(-1)} \sqcup (\lambda_{(0)})^e \sqcup (\lambda_{(1)})^{e\ell} \sqcup \cdots \sqcup (\lambda_{(n)})^{e\ell^n}.
\end{gather}
In other words,
\begin{gather}\label{eq:paramsimple}
\operatorname{Hom}_{kG}\big(R_{L_{\bf m}}^G(X_{\bf m}), S(\lambda)\big) \simeq D(\boldsymbol\lambda) \qquad \text{in} \ \mathsf{mod}\text{-} \mathcal{H}_{\bf m}. \end{gather}
Conversely, any partition $\lambda$ of $n$ can be uniquely decomposed as \eqref{eq:multipart} using the $e$-$\ell$-adic decomposition def\/ined in Section~\ref{sec:partitions}. The tuple $\mathbf{m} = (|\lambda_{(-1)}|, |\lambda_{(0)}|,\ldots,|\lambda_{(n)}|)$ will be denoted by $\mathsf{hc}(\lambda)$, thus def\/ining a map $\mathsf{hc} \colon \mathcal{P}(n) \longrightarrow \mathcal{N}(n)$. With this notation, the simple $kG$-modules lying in the Harish-Chandra of $(L_{\bf m},X_{\bf m})$ are parametrized by $\mathsf{hc}^{-1}(\mathbf{m})$.

Motivated by the isomorphism \eqref{eq:isohecke} and the tensor product theorem of Dipper--Du we def\/ine a version of the Mullineux involution as follows.

\begin{Definition}\label{def:mullineux}
Let $\lambda$ be a partition of $n$, written as $\lambda = \lambda_{(-1)} \sqcup (\lambda_{(0)})^e \sqcup \cdots \sqcup (\lambda_{(n)})^{e\ell^n}$ where $\lambda_{(-1)}$ is $e$-regular and each $\lambda_{(i)}$ for $i \geq 0$ is $\ell$-regular.
The \emph{generalized Mullineux involution} on $\lambda$ is def\/ined by
\begin{gather*} \mathsf{M}_{e,\ell} (\lambda) = \mathsf{M}_{e}(\lambda_{(-1)}) \sqcup ( \mathsf{M}_{\ell}(\lambda_{(0)}))^e \sqcup \cdots \sqcup ( \mathsf{M}_{\ell}(\lambda_{(n)}))^{e\ell^n}.\end{gather*}
\end{Definition}

\begin{Example}\label{ex:mull}It is interesting to note the following particular cases.
\begin{enumerate}\itemsep=0pt
\item[(a)] If $\lambda$ is an $e$-regular partition then $\Mull_{e,\ell}(\lambda)=\Mull_e(\lambda)$ is the ordinary Mullineux involution.
\item[(b)] If $e\ell > n$ then $\lambda_{(i)} = \varnothing$ for $i> 0$. In addition, $\ell$ is bigger than the size of
$\lambda_{(0)}$ and therefore $\Mull_\ell(\lambda_{(0)}) = \lambda_{(0)}^t$ is just the conjugate of $\lambda_{(0)}$.
Consequently we get
\begin{gather*}\Mull_{e,\ell}(\lambda) = \Mull_e(\lambda_{(-1)}) \sqcup ({\lambda_{(0)}}^t)^e.\end{gather*}
Quite remarkably, this involution already appears in the work of Bezrukavnikov~\cite{Be} and Losev \cite[Corollary~5.7]{Lo} on the wall-crossing bijections for representations of rational Cherednik algebras.
\end{enumerate}
\end{Example}

Let $\operatorname{Irr}_k \big(G | (L_\mathbf{m},X_{\mathbf{m}})\big)$ be the set of isomorphism classes of simple $kG$-modules lying in the Harish-Chandra series associated with $(L_\mathbf{m},X_{\mathbf{m}})$. Note that $M_{e,\ell}$ preserves $\mathsf{hc}^{-1}(\mathbf{m})$, therefore~$S(\lambda)$ and $S(\mathsf{M}_{e,\ell}(\lambda))$ lie in the same Harish-Chandra series. By construction, the involution~$\mathsf{M}_{e,\ell}$ is the unique operation which makes the following diagram commutative:
\begin{gather*}\xymatrix@!C=2.5cm{ \mathsf{hc}^{-1}(\mathbf{m}) \ar@{<->}[r] \ar[d]^{\mathsf{M}_{e,\ell}} & \operatorname{Irr}_k \big(G | (L_\mathbf{m},X_{\mathbf{m}})\big) \ar[rr]^{\operatorname{Hom}_G\big(\HC_{L_\mathbf{m}}^G(X_\mathbf{m}),-\big)} & & \operatorname{Irr} \mathcal{H}_\mathbf{m} \ar[d]^{\alpha^*} \\
 \mathsf{hc}^{-1}(\mathbf{m}) \ar@{<->}[r] & \operatorname{Irr}_k \big(G | (L_\mathbf{m},X_{\mathbf{m}})\big) \ar[rr]^{\operatorname{Hom}_G\big(\HC_{L_\mathbf{m}}^G(X_\mathbf{m}),-\big) }
& & \operatorname{Irr} \mathcal{H}_\mathbf{m}. }\end{gather*}

\subsection{Computation of $\mathsf{d}_G(S)$}
Fix ${\bf m} = (m_{-1},m_0,\ldots,m_n) \in \mathcal{N}(n)$ and let $(L_{\mathbf{m}},X_{\mathbf{m}})$ be the corresponding cuspidal pair, as def\/ined in Section~\ref{sec:hcgln}. One cannot apply Theorem \ref{thm:main} directly to $G$ and $L_{\mathbf{m}}$ since the assumption (ii) might not be satisf\/ied for every intermediate Levi subgroup between $L_{\mathbf{m}}$ and $G$. For example, if ${\bf m} = (e,1)$ then the Levi subgroups $(\mathrm{GL}_1(q))^e \times \mathrm{GL}_e(q)$ and $ \mathrm{GL}_e(q) \times (\mathrm{GL}_1(q))^e$ are conjugate under $\mathrm{GL}_{2e}(q)$ but not under $\mathrm{GL}_e(q) \times \mathrm{GL}_e(q)$. To solve this problem we consider, instead of $\mathbf{G}$, the standard Levi subgroup
\begin{gather*} \mathbf{G}_\mathbf{m} = \mathrm{GL}_{m_{-1}} \big(\overline{\mathbb{F}}_p\big)\times \mathrm{GL}_{em_{0}} \big(\overline{\mathbb{F}}_p\big) \times \mathrm{GL}_{e\ell m_{1}} \big(\overline{\mathbb{F}}_p\big) \times \cdots \times \mathrm{GL}_{e\ell^nm_{n}} \big(\overline{\mathbb{F}}_p\big).\end{gather*}
Then $\mathbf{L}_\mathbf{m}$ is the only standard Levi subgroup of $ \mathbf{G}_m$ which is conjugate to $\mathbf{L}_\mathbf{m}$. In particular, assumption~(ii) of Theorem~\ref{thm:main} is satisf\/ied.

\begin{Lemma}\label{lem:heckeiso}
The functor $\HC_{G_\mathbf{m}}^G$ induces an isomorphism of algebras
\begin{gather*} \operatorname{End}_{G_\mathbf{m}}\big(\HC_{L_\mathbf{m}}^{G_\mathbf{m}}(X_\mathbf{m})\big) \simto
\operatorname{End}_{G}\big(\HC_{L_\mathbf{m}}^{G}(X_\mathbf{m})\big) = \mathcal{H}_\mathbf{m}.\end{gather*}
\end{Lemma}

\begin{proof} Since $\HC_{G_\mathbf{m}}^G$ is fully-faithful, the natural map $\operatorname{End}_{G_\mathbf{m}}(\HC_{L_\mathbf{m}}^{G_\mathbf{m}}(X_\mathbf{m})) \longrightarrow
\operatorname{End}_{G}(\HC_{L_\mathbf{m}}^{G}(X_\mathbf{m}))$ is an embedding of algebras. To conclude it is enough to compute the dimensions using the Mackey formula. The equality comes from the fact that any element $g \in G$ which normalizes~$L_{\mathbf{m}}$ and~$X_\mathbf{m}$ is in fact in $G_\mathbf{m}$.
\end{proof}

In particular, the simple $kG_\mathbf{m}$-modules lying in the Harish-Chandra series of $(L_\mathbf{m},X_{\mathbf{m}})$ are also parametrized by multipartitions $\boldsymbol\lambda = (\lambda_{(-1)},\lambda_{(0)},\ldots,\lambda_{(n)})$ of $\mathbf{m}$, where $\lambda_{(-1)}$ is $e$-regular, and each $\lambda_{(i)}$ for $i\geq 0$ is $\ell$-regular. We will write $S_\mathbf{m}(\boldsymbol\lambda)$ for the simple module corresponding to $\boldsymbol\lambda$, which by def\/inition satisf\/ies
\begin{gather}\label{eq:parammulti}
\operatorname{Hom}_{kG_{\bf m}}(R_{L_{\bf m}}^{G_{\bf m}}(X_{\bf m}), S_{\bf m}(\boldsymbol \lambda)) \simeq D(\boldsymbol\lambda) \qquad \text{in} \ \mathsf{mod}\text{-} \mathcal{H}_{\bf m} \end{gather}
(compare with~\eqref{eq:paramsimple}). Then it follows from the tensor product theorem \cite[Corolary~4.3.11]{DipDu} that
\begin{gather*} S_{\bf m}(\boldsymbol\lambda) = S(\lambda_{(-1)}) \boxtimes S\big((\lambda_{(0)})^e\big) \boxtimes S\big((\lambda_{(1)})^{e\ell}\big) \boxtimes \cdots \end{gather*}
and
\begin{gather} \label{eq:tensorprod}
R_{G_{\bf m}}^G(S_{\bf m}(\boldsymbol\lambda)) \simeq S(\lambda),
\end{gather}
where as in \eqref{eq:multipart} we set $\lambda = \lambda_{(-1)} \sqcup (\lambda_{(0)})^e \sqcup (\lambda_{(1)})^{e\ell} \sqcup \cdots \sqcup (\lambda_{(n)})^{e\ell^n}$.

The construction of the isomorphism \eqref{eq:isohecke} given for example in \cite[Section~19]{CaEn} is compatible with induction and restriction. The map $\mathbf{M} \longmapsto \operatorname{End}_{M}(\HC_{L_\mathbf{m}}^{M}(X_\mathbf{m}))$ gives a one-to-one correspondence between the standard Levi subgroups of $\mathbf{G}_m$ containing $\mathbf{L}$ and the parabolic subalgebras of $\mathcal{H}_\mathbf{m}$. Since $q \neq 0$, $\mathcal{H}_\mathbf{m}$ is f\/lat over each of these subalgebras and Theorem \ref{thm:main} can be applied to get
\begin{gather}\label{eq:dhanddg}
\mathsf{D}_{\mathcal{H}_\mathbf{m}}\big(\mathrm{RHom}_{G_\mathbf{m}}\big(\HC_{L_\mathbf{m}}^{G_\mathbf{m}}(X_\mathbf{m}),-\big)\big) \, \mathop{\longrightarrow}\limits^\sim \, \mathrm{RHom}_{G_\mathbf{m}}\big(\HC_{L_\mathbf{m}}^{G_\mathbf{m}}(X_\mathbf{m}),\mathsf{D}_{G_\mathbf{m}}(-)\big).
\end{gather}
Recall from Section~\ref{sec:acduality} that given a simple $kG$-module $S$, there is unique composition factor in the cohomology of $\mathsf{D}_G(S)$ which lies in the same Harish-Chandra series as $S$. We denote this composition factor by $\mathsf{d}_G(S)$. Combining \eqref{eq:dhanddg} and Theorem \ref{thm:dualityhecke} we can determine $\mathsf{d}_{G_\mathbf{m}}$ explicitly on the unipotent representations.

\begin{Proposition}\label{prop:thmforGm}
Let $\boldsymbol\lambda = (\lambda_{(-1)}, \lambda_{(0)}, \lambda_{(1)}, \ldots, \lambda_{(n)})$ be a multipartition
where $\lambda_{(-1)}$ is an $e$-regular partition of $m_{-1}$ and each $\lambda_{(i)}$ for $i\geq 0$ is an $\ell$-regular partition of $m_i$. Then
\begin{gather*}d_{G_\mathbf{m}}(S_\mathbf{m}(\boldsymbol\lambda)) \simeq S_\mathbf{m}\big(\mathsf{M}_e(\lambda_{(-1)}), \mathsf{M}_\ell(\lambda_{(0)}), \ldots, \mathsf{M}_\ell(\lambda_{(n)})\big).\end{gather*}
\end{Proposition}

\begin{proof} For simplicity we will write $S$ for $S_\mathbf{m}(\boldsymbol\lambda)$ throughout this proof.
Let $r$ be the cuspidal depth of the $kG_\mathbf{m}$-module $S$. By Proposition \ref{prop:perverse} and the def\/inition of $\mathsf{d}_{G_\mathbf{m}}(S)$ the natural map
\begin{gather*}\operatorname{Hom}_{G_\mathbf{m}}\big(\HC_{L_\mathbf{m}}^{G_\mathbf{m}}(X_\mathbf{m}),\mathsf{d}_{G_\mathbf{m}}(S)\big) \longrightarrow \operatorname{Hom}_{G_\mathbf{m}}\big(\HC_{L_\mathbf{m}}^{G_\mathbf{m}}(X_\mathbf{m}),H^{-r}(\mathsf{D}_{G_\mathbf{m}}(S))\big)\end{gather*}
is an isomorphism of right $\mathcal{H}_\mathbf{m}$-modules. Now let us consider the distinguished triangle
\begin{gather*} H^{-r}(\mathrm{D}_{G_\mathbf{m}}(S)) \longrightarrow \mathsf{D}_{G_\mathbf{m}}(S)[-r] \longrightarrow
\tau_{>-r}(\mathsf{D}_{G_\mathbf{m}}(S))[-r] \rightsquigarrow \end{gather*}
in $D^b(kG_\mathbf{m})$. We apply the functor $\operatorname{Hom}_{D^b(kG_\mathbf{m})}\big(\HC_{L_\mathbf{m}}^{G_\mathbf{m}}(X_\mathbf{m}),-\big)$, which, by the properties of $\mathsf{D}_{G_\mathbf{m}}(S)$ listed in Proposition \ref{prop:perverse}, gives an isomorphism
\begin{gather*}\operatorname{Hom}_{G_\mathbf{m}}\big(\HC_{L_\mathbf{m}}^{G_\mathbf{m}}(X_\mathbf{m}),H^{-r}(\mathrm{D}_{G_\mathbf{m}}(S))\big) \, \mathop{\longrightarrow}\limits^\sim \, \operatorname{Hom}_{D^b(kG_\mathbf{m})}\big(\HC_{L_\mathbf{m}}^{G_\mathbf{m}}(X_\mathbf{m}),\mathsf{D}_{G_\mathbf{m}}(S)[-r]\big).\end{gather*}
Combining this with Theorem \ref{thm:main} gives
\begin{gather*}
\operatorname{Hom}_{G_\mathbf{m}}\big(\HC_{L_\mathbf{m}}^{G_\mathbf{m}}(X_\mathbf{m}),H^{-r}(\mathrm{D}_{G_\mathbf{m}}(S))\big) \simeq H^{-r}\big(\mathrm{RHom}_{G_\mathbf{m}}\big(\HC_{L_\mathbf{m}}^{G_\mathbf{m}}(X_\mathbf{m}),\mathsf{D}_{G_\mathbf{m}}(S)\big)\big) \\
\hphantom{\operatorname{Hom}_{G_\mathbf{m}}\big(\HC_{L_\mathbf{m}}^{G_\mathbf{m}}(X_\mathbf{m}),H^{-r}(\mathrm{D}_{G_\mathbf{m}}(S))\big) }{}
\simeq H^{-r}\big(\mathsf{D}_\mathcal{H}\big(\mathrm{RHom}_{G_\mathbf{m}}(\HC_{L_\mathbf{m}}^{G_\mathbf{m}}(X_\mathbf{m}),S)\big)\big)
\end{gather*}
in $\modn\mathcal{H}_\mathbf{m}$. Now, by Theorem~\ref{thm:dualityhecke} the duality functor $\mathsf{D}_\mathcal{H}$ is induced by a shifted Morita equivalence, obtained by twisting by the algebra automorphism $\alpha$ def\/ined in Section~\ref{sec:hecke}. Note that the corresponding shift equals the rank of the Coxeter group associated with the Hecke algebra $\mathcal{H}_\mathbf{m}$, which also equals the cuspidal depth $r$ of~$S$. In particular, we have
\begin{gather*}
H^{-r}\big(\mathsf{D}_\mathcal{H}\big(\mathrm{RHom}_{G_\mathbf{m}}(\HC_{L_\mathbf{m}}^{G_\mathbf{m}}(X_\mathbf{m}),S)\big)\big)
\simeq H^{-r} \big(\alpha^* \mathrm{RHom}_{G_\mathbf{m}}\big(\HC_{L_\mathbf{m}}^{G_\mathbf{m}}(X_\mathbf{m}),S\big)[r]\big) \\
\hphantom{H^{-r}\big(\mathsf{D}_\mathcal{H}\big(\mathrm{RHom}_{G_\mathbf{m}}(\HC_{L_\mathbf{m}}^{G_\mathbf{m}}(X_\mathbf{m}),S)\big)\big)}{}
\simeq \alpha^*H^{0} \big( \mathrm{RHom}_{G_\mathbf{m}}\big(\HC_{L_\mathbf{m}}^{G_\mathbf{m}}(X_\mathbf{m}),S\big)\big) \\
\hphantom{H^{-r}\big(\mathsf{D}_\mathcal{H}\big(\mathrm{RHom}_{G_\mathbf{m}}(\HC_{L_\mathbf{m}}^{G_\mathbf{m}}(X_\mathbf{m}),S)\big)\big)}{}
\simeq \alpha^*\big(\operatorname{Hom}_{D^b(kG_\mathbf{m})}\big(\HC_{L_\mathbf{m}}^{G_\mathbf{m}}(X_\mathbf{m}),S\big)\big)
\simeq \alpha^* D(\boldsymbol\lambda).
\end{gather*}
Note that the last isomorphism uses \eqref{eq:parammulti} and the fact that the natural functor $kG_{\bf m} \rightarrow D^b(kG_{\bf m})$ is fully-faithful. Finally, by def\/inition of the Mullineux involution, the module $\alpha^* D(\boldsymbol\lambda)$ is the simple $\mathcal{H}_\mathbf{m}$-module labelled by the multipartition $(\mathsf{M}_e(\lambda_{(-1)}), \mathsf{M}_\ell(\lambda_{(0)}), \ldots, \mathsf{M}_\ell(\lambda_{(n)}))$.
\end{proof}

We can f\/inally prove the expected relation between the Alvis--Curtis duality and our gene\-ralization of the Mullineux involution.

\begin{Theorem}\label{thm:main2} Let $e = \min \{i\geq 0\, |\, 1+q+q^2+\cdots+q^{i-1}\equiv 0 \, {\rm mod}\, \ell\}$ and $\lambda$ be a partition of~$n$. Then
\begin{gather*} \mathsf{d}_{\mathrm{GL}_n(q)} (S(\lambda)) \simeq S(\mathsf{M}_{e,\ell}(\lambda)),\end{gather*}
where $\mathsf{M}_{e,\ell}$ is the generalized Mullineux involution defined in Definition~{\rm \ref{def:mullineux}}.
\end{Theorem}

\begin{proof} Let $\boldsymbol\lambda$ be the multipartition associated to $\lambda$ as in \eqref{eq:multipart}.
Given a bounded complex~$C$ of representations (of $kG$ or $kG_\mathbf{m}$), recall that $[C]$ denotes its class in the corresponding Grothen\-dieck group ($K_0(kG)$ or $K_0(kG_\mathbf{m})$, see Section~\ref{sec:notation}). By~\eqref{eq:curtistype} we have
\begin{gather}\label{eq:grothgroups}
\big[\mathsf{D}_G\big(\HC_{G_\mathbf{m}}^G(S_\mathbf{m}(\boldsymbol\lambda))\big)\big] = \pm \big[\HC_{G_\mathbf{m}}^G (\mathsf{D}_{G_\mathbf{m}}(S_\mathbf{m}(\boldsymbol\lambda)))\big] \qquad \text{in} \ K_0(kG).
\end{gather}
Let $\mathcal{C}$ (resp.~$\mathcal{C}_\mathbf{m}$) be the sublattice of $K_0(kG)$ (resp.~$K_0(kG_\mathbf{m})$) spanned by the classes of simple modules with cuspidal depth strictly less than $S(\lambda)$ (resp.~$S_\mathbf{m}(\boldsymbol\lambda)$). It follows from Proposition~\ref{prop:perverse}(ii) that these lattices are stable under Alvis--Curtis duality. In addition, the Harish-Chandra induction functor satisf\/ies $\HC_{G_\mathbf{m}}^G(\mathcal{C}_\mathbf{m}) \subset \mathcal{C}$.

By Proposition \ref{prop:perverse}, we have $[\mathsf{D}_{G_\mathbf{m}}(S_\mathbf{m}(\boldsymbol\lambda))] \in \pm [\mathsf{d}_{G_\mathbf{m}}(S_{\mathbf{m}}(\boldsymbol\lambda))] + \mathcal{C}_\mathbf{m}$. We deduce from~\eqref{eq:tensorprod} and Proposition~\ref{prop:thmforGm} that
\begin{gather*}\big[\HC_{G_\mathbf{m}}^G (\mathsf{D}_G(S_\mathbf{m}(\boldsymbol\lambda)))\big] \in \pm[S(\mathsf{M}_{e,\ell}(\lambda))]+ \mathcal{C}.\end{gather*}
On the other hand $[\mathsf{D}_G(S(\lambda))] \in \pm [\mathsf{d}_G(S(\lambda))] + \mathcal{C}$, so that again by \eqref{eq:tensorprod} we have
\begin{gather*} \big[\mathsf{D}_G\big(\HC_{G_\mathbf{m}}^G(S_\mathbf{m}(\boldsymbol\lambda))\big)\big] \in \pm[\mathsf{d}_G(S(\lambda))]+ \mathcal{C}\end{gather*}
and we conclude that $[\mathsf{d}_G(S(\lambda))] = [S(\mathsf{M}_{e,\ell}(\lambda))]$ using~\eqref{eq:grothgroups} .
\end{proof}

\begin{Remark}The simple unipotent $k\mathrm{GL}_n(q)$-module associated with the trivial partition \smash{$\lambda{=}(n)$} is the trivial module $k$. In that case the complex $\mathsf{D}_G(k)$ is quasi-isomophic to a module shifted in degree $-n+1$, by the Solomon--Tits theorem \cite[Theorem~66.33]{CuRei}. This module is a~charac\-te\-ris\-tic~$\ell$ version of the Steinberg representation. By Theorem~\ref{thm:main2}, its socle is isomorphic to~$S(\mathsf{M}_e(n))$, which is consistent with~\cite{Geck}.
\end{Remark}

\section{Interpretation in terms of crystals}

The aim of this section is to give an alternative description of the map $\mathsf{M}_{e,\ell}$ using the crystal graph theory in the same spirit as for the original Mullineux involution~\cite{Kle} (see Proposi\-tion~\ref{prop:main3}).

\subsection{More on partitions}
We f\/ix an integer $d > 1$. Given a partition $\lambda=(\lambda_1\geq \lambda_2 \geq \cdots \geq \lambda_r > 0)$, its \emph{Young diagram} $[\lambda]$ is the set
\begin{gather*}[\lambda] = \big\{(a,b)\, |\, 1\leq a\leq r,\, 1 \leq b\leq \lambda_a\big\} \subset \mathbb{N}\times\mathbb{N}.\end{gather*}
The elements of this set are called the {\it nodes} of $\lambda$. The {\it $d$-residue} of a node $\gamma\in [\lambda]$ is by def\/inition $\operatorname{res}_d (\gamma)=b-a+d\mathbb{Z}$. For $j\in \mathbb{Z}/d\mathbb{Z}$, we say that $\gamma$ is a {\it $j$-node} if $\operatorname{res}_d (\gamma)=j$. In addition, $\gamma$~is called a~{\it removable $j$-node} for $\lambda$ if the set $[\lambda]\setminus\{\gamma\} $ is the Young diagram of some partition $\mu$. In this case, we also say that $\gamma$ is an {\it addable $j$-node} for $\mu$. We write $ \mu \overset{j}{\rightarrow} \lambda$ if $[\mu]\subset [\lambda]$ and $[\lambda]\setminus [\mu]=\{\gamma\}$ for a $j$-node $\gamma$.

Let $\gamma=(a,b)$ and $\gamma'=(a',b')$ be two addable or removable $j$-nodes of the same partition $\lambda$. Then we write $\gamma>\gamma'$ if $a<a'$. Let $w_j (\lambda)$ be the word obtained by reading all the addable and removable $j$-nodes in increasing order and by encoding each addable $j$-node with the letter $A$ and each removable $j$-node with the letter $R$. Then deleting as many subwords $RA$ in this word as possible, we obtain a sequence $A\cdots A R \cdots R$. The node corresponding to the rightmost $A$ (if it exists) is called the {\it good addable $j$-node} and the node corresponding to the leftmost~$R$ (if it exists) is called the {\it good removable $j$-node}.

\subsection{Fock space and Kashiwara operators}\label{fock}
 Let $\mathcal{F} := \mathbb{C} \mathcal{P}$ be the $\mathbb{C}$-vector space with basis given by the set $\mathcal{P}$ of all
 partitions. There is an action of the quantum group $\mathcal{U} (\widehat{\mathfrak{sl}}_d)$ on $\mathcal{F}$ \cite{MM} which makes $\mathcal{F}$ into an integrable module of level~$1$. The Kashiwara operators $\widetilde{E}_{i,d}$ and $\widetilde{F}_{i,d}$ are then def\/ined as follows:
\begin{gather*}\widetilde{F}_{i,d} \cdot\lambda= \begin{cases}
 \mu & \textrm{if } \lambda \overset{i}{\rightarrow} \mu \text{ and } [\mu]\setminus[\lambda] \text{ is a good addable $i$-node of $\lambda$},\\
 0 & \text{if $\lambda$ has no good addable $i$-node},\end{cases}
\\
\widetilde{E}_{i,d} \cdot \mu= \begin{cases}
 \lambda & \textrm{if } \lambda \overset{i}{\rightarrow} \mu \text{ and } [\mu]\setminus[\lambda] \text{ is a good removable $i$-node of $\mu$},\\
 0 & \textrm{if $\mu$ has no good removable $i$-node}.
 \end{cases}\end{gather*}
Using these operators one can construct the $\widehat{\mathfrak{sl}}_d$-{\it crystal graph} of $\mathcal{F}$, which is the graph with
\begin{itemize}\itemsep=0pt
 \item vertices: all the partitions $\lambda$ of $n\in \mathbb{N}$,
 \item arrows: there is an arrow from $\lambda$ to $\mu$ colored by $i\in \mathbb{Z}/d\mathbb{Z}$
 if and only if $\widetilde{F}_{i,d} \cdot \lambda=\mu$, or equivalently if and only if
 $\lambda=\widetilde{E}_{i,d} \cdot \mu$.
\end{itemize}
Note that the def\/inition makes sense for $d=\infty$. The corresponding ${\mathfrak{sl}}_\infty$-crystal graph coincides with the Young graph, also known as the branching graph of the complex irreducible representations of symmetric groups.

The following result can be found for example in \cite[Section~2.2]{LLT}.

\begin{Proposition}\label{regk} A partition $\lambda$ is a $d$-regular partition of $n$ if and only if there exists $(i_1,\ldots,i_n)$ $\in (\mathbb{Z}/d\mathbb{Z})^n$ such that
\begin{gather*} \widetilde{F}_{i_1,d}\cdots \widetilde{F}_{i_n,d}\cdot \varnothing=\lambda .\end{gather*}
\end{Proposition}

In other words, the connected component the $\widehat{\mathfrak{sl}}_d$-crystal graph containing the empty partition is the full subgraph of the $\widehat{\mathfrak{sl}}_d$-crystal whose vertices are labelled by $d$-regular partitions.

The arrows in this component give the branching rule for induction and restriction in the Hecke algebra of symmetric groups at a primitive $d$-th root of unity (see \cite{Brun} for more details). The partitions $\lambda$ for which we have $\widetilde{E}_{i,d} \cdot \lambda=0$ for all $i\in \mathbb{Z}/d\mathbb{Z}$ are the \emph{highest weight vertices}. One can observe that they correspond to partitions of the form $\lambda=\mu^d$ for some partition~$\mu$.

There is also a representation theoretic interpretation of the other components, using the representation theory of the f\/inite general linear group. Let $e$ be the order of $q$ modulo $\ell$, which we assume to be dif\/ferent from $1$.
Following~\cite{GHJ}, one can def\/ine a \emph{weak Harish-Chandra theory} for unipotent representations of $\mathrm{GL}_n(q)$ for various $n$. Recall from Section~\ref{sec:hcgln} that these unipotent representations are parametrized by partitions. Consequently, the complexif\/ied Grothendieck group of the category of unipotent representations is naturally isomophic to $\mathcal{F}$. Under this identif\/ication, it follows from \cite{ChRou08} that the action of $\mathcal{U} (\widehat{\mathfrak{sl}}_e)$ comes from a truncated version of Harish-Chandra induction and restriction. As in \cite{DVV}, we deduce that:
\begin{itemize}\itemsep=0pt
\item A simple unipotent module $S(\lambda)$ is weakly cuspidal if and only if $\lambda$ labels a highest weight vertex in the $\widehat{\mathfrak{sl}}_e$-crystal graph of $\mathcal{F}$.
\item Two simple modules $S(\lambda)$ and $S(\mu)$ of $k\mathrm{GL}_n(q)$ lie in the same weak Harish-Chandra series if and only if
 there exist a highest weight vertex $\nu$, $k\in \mathbb{N}$, $(i_1,\ldots,i_k)\in (\mathbb{Z}/e\mathbb{Z})^k$
 and $(j_1,\ldots,j_k)\in (\mathbb{Z}/e\mathbb{Z})^k$ such that
 \begin{gather*}\widetilde{F}_{i_1,e}\ldots \widetilde{F}_{i_k,e}\cdot\nu=\lambda,\qquad \text{and} \qquad
 \widetilde{F}_{j_1,e}\ldots \widetilde{F}_{j_k,e}\cdot\nu=\mu.\end{gather*}
 This means that $\lambda$ and $\mu$ are in the same connected component of the associated $\widehat{\mathfrak{sl}}_e$-crystal graph.
\end{itemize}

\subsection[Crystals: the case $\ell=\infty$]{Crystals: the case $\boldsymbol{\ell=\infty}$}
Let $e\in \mathbb{Z}_{>1}$. Recall from Section~\ref{sec:partitions} that any partition $\lambda$ of $n$ can be decomposed in a unique way as
\begin{gather*}\lambda=\lambda_{(-1)} \sqcup (\lambda_{(0)})^e,\end{gather*}
where $\lambda_{(-1)}$ is $e$-regular.

We claim that the entire $\widehat{\mathfrak{sl}}_e$-crystal graph structure on the Fock space may be recovered from
 the subgraph with vertices labelled by $e$-regular partitions. Indeed,
it follows from the def\/inition of the Kashiwara operators that for any partition $\lambda$
\begin{gather}\label{eq:ecrystal}
\widetilde{F}_{i,e} \lambda=\mu\iff \widetilde{F}_{i,e} \lambda_{(-1)}=\mu_{(-1)} \qquad \text{and} \qquad \lambda_{(0)}=\mu_{(0)}.
\end{gather}
One can now def\/ine crystal operators $\widetilde{F}_{i,\infty,0}$ and $\widetilde{E}_{i,\infty,0}$ for all $i\in \mathbb{Z}$ by
\begin{gather*}
\widetilde{F}_{i,\infty,0} \lambda:=\mu \iff \widetilde{F}_{i,\infty} \lambda_{(0)}=\mu_{(0)} \qquad \text{and} \qquad \lambda_{(-1)}=\mu_{(-1)}, \\
\widetilde{E}_{i,\infty,0} \lambda:=\mu \iff \widetilde{E}_{i,\infty} \lambda_{(0)}=\mu_{(0)} \qquad \text{and} \qquad \lambda_{(-1)}=\mu_{(-1)}.
\end{gather*}
This endows $\mathcal{P}$ with an ${\mathfrak{sl}}_{\infty}$-crystal structure, which by~\eqref{eq:ecrystal} commutes with the $\widehat{\mathfrak{sl}}_e$-crystal structure. Note that the only highest weight with respect to these two structures is the empty partition. In fact, these constructions already appear in the work of Losev \cite{Lo} (where the ${\mathfrak{sl}}_{\infty}$-crystal is called the \emph{Heisenberg crystal}, see also \cite[Proposition~4.6]{Lo2}).

\subsection[Crystals: the case $\ell\in \mathbb{N}$]{Crystals: the case $\boldsymbol{\ell\in \mathbb{N}}$}
More generally, recall from Section~\ref{sec:hcgln} that any partition $\lambda$ of $n$ can be decomposed in a unique way as
\begin{gather}\label{eq:decomplambdabis}
\lambda = \lambda_{(-1)} \sqcup (\lambda_{(0)})^e \sqcup (\lambda_{(1)})^{e\ell} \sqcup \cdots \sqcup (\lambda_{(n)})^{e\ell^n},
\end{gather}
where $\lambda_{(-1)}$ is $e$-regular and each $\lambda_{(i)}$ for $i>-1 $ is $\ell$-regular. For example, with $\mu=(2^2.1^7)$, $e=2$ and $\ell=3$ we obtain $\mu= (1)\sqcup (2)^2 \sqcup (1)^{2\times 3}$.

As in the previous section, the $\widehat{\mathfrak{sl}}_e$-crystal operators act on the $e$-regular part of partitions, which makes the def\/inition of other operators possible. Let us f\/ix $j\in \mathbb{N}$. We set $\widetilde{F}_{i,\ell,j} \lambda=\mu$ if and only if $\widetilde{F}_{i,\ell} \lambda_{(j)}=\mu_{(j)}$ and $\lambda_{(l)} = \mu_{(l)}$ for all $l \neq j$. Similarly, $\widetilde{E}_{i,\ell,j} \lambda=\mu$ if and only if $\widetilde{E}_{i,\ell} \lambda_{(j)}=\mu_{(j)}$ and $\lambda_{(l)}=\mu_{(l)}$ for all $l\neq j$. In other words, $\widetilde{F}_{i,\ell,j}$ and
$\widetilde{E}_{i,\ell,j}$ are def\/ined as the usual $\widehat{\mathfrak{sl}}_\ell$-crystal operators acting on the component $\lambda_{(j)}$ in the decomposition~\eqref{eq:decomplambdabis}, or equivalently on the (non-necessarily $\ell$-regular) partition given by $\lambda_{(j)}\sqcup \lambda_{(j+1)}^{\ell} \sqcup \cdots$.

As a consequence we obtain an $\widehat{\mathfrak{sl}}_e$-crystal structure of level $1$ together with many $\widehat{\mathfrak{sl}}_\ell$-crystal structures (each of them indexed by an integer $j \in \mathbb{N}$, and of level $e\ell^j$) on the set of partitions. The following proposition is clear using the decomposition of a partition.

\begin{Proposition}The above $\widehat{\mathfrak{sl}}_e$-crystal structure and the various $\widehat{\mathfrak{sl}}_\ell$-crystal structures
on $\mathcal{P}$ mutually commute.
\end{Proposition}

We can thus def\/ine a graph containing all the information on the $\widehat{\mathfrak{sl}}_e$-crystal structure and $\widehat{\mathfrak{sl}}_\ell$-crystal structures. There is then an obvious notion of highest weight.

\begin{Lemma}\label{lem:hw} The empty partition is the unique highest weight vertex with respect to the above crystal structure.
\end{Lemma}

In other words, the corresponding crystal graph is connected.

\begin{proof}
Assume that $\lambda$ is a non-empty partition. Then there exists $r\in\mathbb{Z}_{\geq -1}$
 such that $\lambda_{(r)}\neq \varnothing$. Now
 \begin{itemize}\itemsep=0pt
 \item If $r=-1$ then $\lambda_{-1}$ is $e$-regular and we have $\widetilde{E}_{i,e}\cdot \lambda\neq \varnothing$
 for some $i\in \mathbb{Z}$.
 \item If $r\neq-1$ then $\lambda_{r}$ is $\ell$-regular and we have $\widetilde{E}_{i,\ell,r}\cdot \lambda\neq \varnothing$
 for some $i\in \mathbb{Z}$.
\end{itemize}
Thus $\lambda$ is not a highest weight vertex, and $\varnothing$ is the only highest weight vertex.
\end{proof}

\begin{Example}Let $e=2$ and consider the partition $\lambda=(2^2.1^7)$.
\begin{enumerate}\itemsep=0pt
\item[(a)] Assume that $\ell=\infty$. Then we have $\lambda_{(-1)}=1$ and $\lambda_{(0)}=2.1^3$. We have
\begin{gather*}\widetilde{F}_{2,\infty,0} \cdot \lambda:=(1)\sqcup \big(3.1^3\big)^2=3^2. 1^7,\\
\widetilde{F}_{0,\infty,0} \cdot \lambda:=(1)\sqcup \big(2^2.1^2\big)^2=2^4 .1^5,\\
\widetilde{F}_{-4,\infty,0}\cdot \lambda:=(1)\sqcup \big(2.1^4\big)^2=2^2 .1^9.\end{gather*}
All the others Kashiwara operators $\widetilde{F}_{j,\infty,0}$ act by $0$ on $\lambda$.

\item[(b)] Assume that $\ell=3$. Then we have $\lambda_{(-1)}=1$, $\lambda_{(0)}=2$ and $\lambda_{(1)}=1$.
\begin{gather*}\widetilde{F}_{2,3,0} \cdot \lambda:=(1)\sqcup (3)^2 \sqcup 1^6 =3^2 .1^7\end{gather*}
and the action of the other operators $\widetilde{F}_{j,0,3}$ are $0$. We also have
\begin{gather*}\widetilde{F}_{1,3,1} \cdot \lambda:=(1)\sqcup (2)^2 \sqcup 2^6 =2^8 .1,\\
\widetilde{F}_{2,3,1}\cdot \lambda:=(1)\sqcup (2)^2 \sqcup (1.1)^6 =2^2 .1^{13}.\end{gather*}
One can also consider
\begin{gather*}\widetilde{F}_{1,3,2} \cdot \lambda=(1)\sqcup (3)^2 \sqcup 1^6\sqcup 1^{18}=3^2 . 1^{25}\end{gather*}
or, more generally for $k\geq 2$:
\begin{gather*}\widetilde{F}_{1,3,k} \cdot \lambda=(1)\sqcup (3)^2\sqcup 1^6 \sqcup 1^{3^k \times 2}=3^2 .1^{2\times 3^k+6}.\end{gather*}
\end{enumerate}
\end{Example}

\begin{Remark}Assume $\lambda$ is a partition of $n$ such that $e\ell>n$. Then it follows from the construction that $\lambda_{(k)}=\varnothing$ for all $k>0$. Note also that $\lambda_{(0)}$ is a partition of rank strictly less than $\ell$, and therefore it is $\ell$-regular. In that case the action of $\widetilde{E}_{i,\ell,0}$ and $\widetilde{E}_{i,\infty,0}$ coincide. In particular, the $\widehat{\mathfrak{sl}}_\ell$ and ${\mathfrak{sl}}_{\infty}$-crystal structures coincide for the partitions of rank less than~$n$.
\end{Remark}

\subsection{Crystals and the Mullineux involution}\label{sec:crystalmull}
Recall that any $d$-regular partition $\lambda$ belongs to the connected component of the empty partition in the $\widehat{\mathfrak{sl}_d}$-crystal graph. The image of $\lambda$ by the Mullineux involution $\Mull_d$ is also in that component and it can be computed using the following theorem.

\begin{Theorem}[Ford--Kleshchev \cite{ForKle}]\label{FK}
 Let $\lambda \in \operatorname{Reg}_d (n)$ and let $(i_1,\ldots,i_n)\in (\mathbb{Z}/d\mathbb{Z})^n$ such that
\begin{gather*} \widetilde{F}_{i_1,d}\cdots \widetilde{F}_{i_n,d} \cdot \varnothing=\lambda \end{gather*}
$($see Proposition~{\rm \ref{regk})}. Then there exists $\mu \in \operatorname{Reg}_d (n)$ such that
\begin{gather*} \widetilde{F}_{-i_1,d}\cdots \widetilde{F}_{-i_n,d} \cdot \varnothing=\mu
\end{gather*}
and $\mu = \Mull_d (\lambda)$.
\end{Theorem}

\begin{Remark}\label{conj}
When $d > n$ every addable node is good, and Theorem \ref{FK} implies that \smash{$M_d(\lambda) {=} \lambda^t$}, the conjugate partition (see also Remark~\ref{rem:asymptotic}).
\end{Remark}

From the def\/inition of $\mathsf{M}_{e,\ell}$ (see Def\/inition \ref{def:mullineux}) and the construction of the various crystal operators, Theorem \ref{FK} generalizes to the following situation.

\begin{Proposition}\label{prop:main3}
Let $\lambda$ be a partition which we write
\begin{gather*}\widetilde{F}_{i_1,p_1,k_1}\cdots \widetilde{F}_{i_m,p_m,k_m} \cdot \varnothing =\lambda\end{gather*}
with for all $j=1,\ldots,m$,
\begin{itemize}\itemsep=0pt
\item $k_j\in \mathbb{Z}_{\geq -1}$,
\item $p_j=e$ if $k_j=-1$ and $p_j=\ell$ otherwise,
\item $i_j \in \mathbb{Z}/p_j \mathbb{Z}$.
\end{itemize}
Then
\begin{gather*}\widetilde{F}_{-i_1,p_1,k_1}\cdots \widetilde{F}_{-i_m,p_m,k_m} \cdot \varnothing =\Mull_{e,\ell}(\lambda).\end{gather*}
\end{Proposition}

\begin{proof}
This is clear as the usual Mullineux involution on the set of $e$-regular partitions is given in
Theorem \ref{FK}.
\end{proof}

\begin{Remark}\label{conj1}
When $\ell=\infty$, the result remains valid and $\Mull_{e,\infty}$ coincides with the operation described by Bezrukavnikov in \cite{Be} and Losev in~\cite{Lo} (see Example~\ref{ex:mull}(b)).
\end{Remark}

\subsection*{Acknowledgements}

The authors gratefully acknowledge f\/inancial support by the ANR grant GeRepMod ANR-16-CE40-0010-01. We thank Gunter Malle, Emily Norton and the referees for their many valuable comments on a preliminary version of the manuscript.

\pdfbookmark[1]{References}{ref}
\LastPageEnding

\end{document}